\documentclass[11pt,a4paper,reqno]{amsart}
\usepackage{amssymb}
\usepackage{amscd}
\usepackage{enumerate}
\usepackage{graphicx}
\usepackage{siunitx}
\usepackage{tikz-cd}
\usepackage{color}
\usepackage[colorlinks,
            linkcolor=black,
            anchorcolor=black,
            citecolor=black
            ]{hyperref}
\usetikzlibrary{arrows}
\numberwithin{equation}{section}

\usepackage{mathabx}

\usepackage{mathtools}
\usepackage[tableposition=top]{caption}
\usepackage{booktabs,dcolumn}

\DeclareFontFamily{OT1}{rsfs}{}
\DeclareFontShape{OT1}{rsfs}{n}{it}{<-> rsfs10}{}
\DeclareMathAlphabet{\mathscr}{OT1}{rsfs}{n}{it}

\theoremstyle{plain}

\newtheorem{theorem}{Theorem}[section]

\newtheorem{lemma}[theorem]{Lemma}

\theoremstyle{definition}

\newtheorem{remark}[theorem]{Remark}

\newcommand{\abs}[1]{\left\vert#1\right\vert}
\newcommand{\norm}[1]{\Vert#1\Vert}

\newcommand\R{\mathbb{R}}

\def \Re {\,{\rm Re}\,}
\def \supp {\,{\rm supp}\,}

\usepackage{algorithm, algorithmic}

\begin{document}

\title[$L^p$ decay of oscillatory integral operators]{Sharp $L^p$ decay estimates for degenerate and singular oscillatory integral operators}

\author{Shaozhen Xu}
\address{Department of Mathematics, Nanjing University, China, 210093}
\email{shaozhen@nju.edu.cn}


\begin{abstract}
We consider the following model of degenerate and singular oscillatory integral operators introduced in \cite{liu1999model}:
\begin{equation}\label{OSO}
Tf(x)=\int_{\R} e^{i\lambda S(x,y)}K(x,y)\psi(x,y)f(y)dy,
\end{equation}
where the phase functions are homogeneous polynomials of degree $n$ and the singular kernel $K(x,y)$ satisfies suitable conditions related to a real parameter $\mu$. We show that the sharp decay estimates on $L^2$ spaces, obtained in \cite{liu1999model}, can be preserved on more general $L^p$ spaces with an additional condition imposed on the singular kernel. In fact, we obtain that
\begin{equation*}
\|Tf\|_{L^p}\leq C_{E,S,\psi,\mu,n,p}\lambda^{-\frac{1-\mu}{n}}\|f\|_{L^p},\ \  \frac{n-2\mu}{n-1-\mu}\leq p \leq\frac{n-2\mu}{1-\mu}.
\end{equation*}
The case without the additional condition is also discussed.
\end{abstract}

\maketitle
\tableofcontents

\section{Introduction}

The central topic of oscillatory integrals is to find the optimal decay estimates. It was revealed that property of phase function plays a dominant role. For nondegenerate scalar oscillatory integrals, stationary phase method gives precise asymptotics which naturally implies the optimal decay estimate. Generally speaking, it is much more complicated to get optimal decay in degenerate cases. For the degenerate oscillatory integrals with real-analytic phases, it was conjectured by Arnold that the optimal decay is determined by the Newton distance of the phase function. By using Hironaka's celebrated theorem of resolution of singularities \cite{hironaka1964resolution}, Varchenko \cite{varchenko1976newton} confirmed Arnold's conjecture for real-analytic phase functions with some nonsingular conditions and thus built deep connections between real algebraic geometry and oscillatory integrals. In harmonic analysis, like Fourier transform, more attentions are paid to those operators with oscillatory kernels. Actually, a general oscillatory integral operator can be written as
\begin{equation}\label{oio}
T_\lambda f(x)=\int_{\R^n} e^{i\lambda S(x,y)}\psi(x,y)f(y)dy.
\end{equation}
If the phase function is nondegenerate in the sense that the Hessian of $S(x,y)$ is nonvanishing in the support of $\psi(x,y)$. By using $TT^{*}$-method globally, H\"{o}rmander \cite{hormander1973oscillatory} showed that $\|T_\lambda\|_{L^2\rightarrow L^2}$ has the sharp decay $\lambda^{-\frac{n}{2}}$. However, for the oscillatory integral operators with degenerate phases, $TT^{*}$-method no longer suits. For polynomial-like phase functions, by using $TT^{*}$-method locally, Phong and Stein successfully established a uniform oscillatory integral estimate, called Operator van der Corput Lemma, details can be found in \cite{phong1994models}\cite{phong1997newton}. By means of the local operator van der Corput lemma and orthogonality argument, Phong and Stein succeeded to give the sharp $L^2$ decay of the oscillatory integral operators with real-analytic phase functions. They also clarified the relation between decay rate and the Newton distance of the phase function. Unlike Varchenko's results, which require some nonsingular conditions on phase functions, Phong-Stein's results in operator setting are verified for all real-analytic phase functions. The difference has an intuitive explanation which has been mentioned in \cite{phong2001multilinear} that operators of the form \eqref{oio} fix the directions of the axes, and this turns out to eliminate the excess freedom in the choice of coordinate systems which was essential in original scalar oscillatory integrals. Based on the pre-mentioned works, Rychkov\cite{rychkov2001sharp} and Greenblatt \cite{greenblatt2005sharp} finally extended similar sharp results to the oscillatory integral operators with smooth phases. Since $L^2$ case has been understood well, an interesting question is whether or not the sharp decay estimate can be preserved on general $L^p$ spaces. In \cite{yang2004sharp}, Yang got an affirmative answer for oscillatory integral operators with special homogeneous polynomial phases. Actually, by embedding these operators into a family of analytic operators and establishing the corresponding boundedness result between critical spaces $H^1$ and $L^1$, he used the damped estimates in \cite{phong1998damped} and Stein's interpolation to give sharp $L^p$ decay. In \cite{shi2017sharp}, Shi and Yan established sharp endpoint $L^p$ decay for arbitrary homogeneous polynomial phase functions. Later, Xiao extended this result to arbitrary analytic phases in \cite{xiao2017endpoint}, another proof see also \cite{shi2019damping}. It should be pointed out that our argument relies heavily on the techniques developed in the previous articles.\\
Inserting a singular kernel into \eqref{oio} will produce \eqref{OSO} which is exactly what this paper is concerned with. Motivated by studying harmonic analysis on nilpotent groups, Ricci and Stein \cite{ricci1987harmonic} first considered this kind of polynomial-phase oscillatory operators with a standard Calder\'{o}n-Zygmund kernel. They established a result that $T$ is bounded from $L^p$ to itself with the bound independent of the coefficients of the phase function, this shows that the decay produced by the oscillatory term is in some sense cancelled by the singular kernel. In \cite{liu1999model}, the author proposed that if some modified size and derivative conditions are imposed on the singular kernel, decay estimates also exist. One direction to extend this result is to study more complicated phases, see for instance \cite{zhang20032}. Another direction is to see if the same sharp decay can be preserved on more general $L^p$ spaces. Our main result tried in the latter direction and states as follows.
\begin{theorem}\label{Main-Result}
For \eqref{OSO}, if $S(x,y)$ is a homogeneous polynomial of degree $n$ and can be written as
\begin{equation}\label{phase}
 S(x,y)=\sum_{k=1}^{n-1}a_kx^{n-k}y^k,
\end{equation}
where $a_1a_{n-1}\neq 0$, $K(x,y)$ is a $C^2$ function away from the diagonal satisfying
\begin{equation}\label{SinCond}
\abs{K(x,y)}\leq E\abs{x-y}^{-\mu}, \ \ \abs{\partial_y^i K(x,y)}\leq E\abs{x-y}^{-\mu-i},
\end{equation}
where the $0<\mu<1$ and $i=1, \text{or } 2$. Moreover, we impose an additional condition on $K(x,y)$,
\begin{equation}\label{AddCondi}
\abs{\partial_x^i K(x,y)}\leq E\abs{x-y}^{-\mu-i}. \tag{AC}
\end{equation}
Then the sharp decay estimate
\begin{equation}\label{Main-Inq}
\|Tf\|_{L^p}\leq C_{E,S,\psi,\mu,n,p} \lambda^{-\frac{1-\mu}{n}}\|f\|_{L^p}
\end{equation}
holds for $\frac{n-2\mu}{n-1-\mu}\leq p\leq \frac{n-2\mu}{1-\mu}$.
\end{theorem}
\begin{remark}
The purpose of imposing the condition \eqref{AddCondi} is to make duality argument in establishing Theorem \ref{Main-Result} and Remark \ref{DualLemma} possible.
\end{remark}
The usual Riesz-Thorin interpolation may give a $L^p$ estimate but it also bears the loss of sharp decay. Applying Stein's complex interpolation method to damped oscillatory integral operators may provide a sharp $L^p$ decay estimate but the proof of boundedness result between critical spaces is complicated in general. Inspired by \cite{xiao2017endpoint}\cite{shi2018uniform}\cite{shi2019damping}, we shall separate the operator into three parts and then apply different interpolations to each one, thus we can get an elegant proof.\\
{\bf Notation:} For two positive constants $C_1$ and $C_2$, $C_1\lesssim C_2$ or $C_1\gtrsim C_2$ means that there exists an constant $C$ independent of $\lambda$ and the test function $f$ such that $C_1\leq CC_2$ or $C_1\geq CC_2$. In this paper, all parameters, some constant depends on, shall be listed in the subscript.

\section{First step toward Theorem \ref{Main-Result}}
Along the lines of the proof of \cite{liu1999model}, \eqref{OSO} is decomposed into two main parts:
\begin{align*}
Tf(x)&=\int_{\R} e^{i\lambda S(x,y)}K(x,y)\psi(x,y)f(y)dy\\
&=\int_{\R} e^{i\lambda S(x,y)}K(x,y)\phi\left((x-y)\lambda^{\frac{1}{n}}\right)\psi(x,y)f(y)dy+\\
&\quad\int_{\R} e^{i\lambda S(x,y)}K(x,y)\left[1-\phi\left((x-y)\lambda^{\frac{1}{n}}\right)\right]\psi(x,y)f(y)dy\\
&:=T_1f(x)+T_2f(x),
\end{align*}
where $\phi\in C_0^\infty(\R)$ and
\begin{equation*}
\phi(x)\equiv
\begin{cases}
0, &\quad \abs{x}\geq 1,\\
1, &\quad \abs{x}\leq \frac{1}{2}.
\end{cases}
\end{equation*}
To prove our Theorem \ref{Main-Result}, it suffices to verify that \eqref{Main-Inq} holds for both $T_1$ and $T_2$. Observe that the kernel of $T_1$
\begin{equation*}
K_1(x,y)=e^{i\lambda S(x,y)}K(x,y)\phi\left((x-y)\lambda^{\frac{1}{n}}\right)\psi(x,y)
\end{equation*}
is absolutely integrable. Thus we may ignore the oscillatory term and employ the following Schur test to prove \eqref{Main-Inq} for $T_1$.
\begin{lemma}\label{Schur-Test}
If the operator
\begin{equation*}
Vf(x)=\int K(x,y)f(y)dy,
\end{equation*}
has a kernel $K(x,y)$ satisfying
\begin{equation*}
\sup_x\int \abs{K(x,y)}dy\leq A_1, \quad \sup_y\int \abs{K(x,y)}dx\leq A_2,
\end{equation*}
then we have
\begin{equation*}
\norm{V}_{L^p\rightarrow L^p}\leq \left(\frac{A_1}{p}+\frac{A_2}{p'}\right),
\end{equation*}
where $1\leq p\leq+\infty$.
\end{lemma}
To make the present paper self-contained, we give the details of the proof.
\begin{proof}
\begin{align*}
\norm{V}_{L^p\rightarrow L^p}&=\sup_{\norm{f}_{L^p}\leq 1}\norm{Vf}_{L^p}\\
&=\sup_{\norm{f}_{L^p}\leq 1}\sup_{\norm{g}_{L^{p'}}\leq 1}\abs{\langle Vf, g\rangle}\\
&\leq \sup_{\norm{f}_{L^p}\leq 1}\sup_{\norm{g}_{L^{p'}}\leq 1}\int\int \abs{K(x,y)}\abs{f}\abs{g}dxdy.
\end{align*}
Convex inequality and Fubini's theorem imply
\begin{align*}
\int\int \abs{K(x,y)}\abs{f(y)}\abs{g(x)}dxdy&\leq \int\int \abs{K(x,y)}\left(\frac{\abs{f(y)}^p}{p}+\frac{\abs{g(x)}^{p'}}{p'}\right)dxdy\\
&\leq\frac{A_1}{p}\int \abs{f(y)}^pdy+\frac{A_2}{p'}\int \abs{g(x)}^pdx.
\end{align*}
It follows $\norm{V}_{L^p\rightarrow L^p}\leq \left(\frac{A_1}{p}+\frac{A_2}{p'}\right)$, then we complete the proof.
\end{proof}
Note that
\begin{equation*}
\sup_x\int \abs{K_1(x,y)}dy\leq C_{E,\psi,\mu}\lambda^{-\frac{1-\mu}{n}}, \quad \sup_y\int \abs{K_1(x,y)}dx\leq C_{E,\psi,\mu}\lambda^{-\frac{1-\mu}{n}}.
\end{equation*}
Thus Lemma \ref{Schur-Test} yields the desired result for $T_1$. For $T_2$, we may suppose $\supp{(\psi)}\subset [-\frac{1}{2}, \frac{1}{2}]\times[-\frac{1}{2}, \frac{1}{2}]$(if not we may impose a dilation on all variables and the dilation factor can be incorporated into $\lambda$). Choose a cut-off function $\Psi\in C_0^{\infty}$ such that $\supp{(\Psi)}\subset[\frac{1}{2},2]$ and $\sum_{l\in \mathbb{Z}}\Psi(2^lx)\equiv 1$. Thus $T_2$ can be dyadically decomposed as
\begin{align*}
T_2f(x)&=\sum_{\sigma_1,\sigma_2=\pm}\sum_{j,k}\int_{\R} e^{i\lambda S(x,y)}K(x,y)\left[1-\phi\left((x-y)\lambda^{\frac{1}{n}}\right)\right]\Psi_j(\sigma_1x)\Psi_k(\sigma_2y)\psi(x,y)f(y)dy\\
&:=\sum_{j,k}T_{j,k}^{\sigma_1,\sigma_2}f(x)
\end{align*}
where $\Psi_j(x)=\Psi(2^jx), \Psi_k(x)=\Psi(2^kx)$. For convenience, we focus only on the case $\sigma_1=+, \sigma_2=+$, the remaining cases can be dealt with similarly. We shall still use $T_2$ and $T_{j,k}$ to denote $\sum_{j,k}T_{j,k}^{+,+}$ and $T_{j,k}^{+,+}$ respectively.
Suppose that
\begin{equation}\label{Hessian}
S^{''}_{xy}(x,y)=c\prod_{l=1}^s(y-\alpha_lx)^{m_l}\prod_{l=1}^rQ_l(x,y),
\end{equation}
where $c$ and $\alpha_j$ are nonzero and each $Q_l(x,y)$ is a positive definite quadratic form. Since we have restricted our attention to the first quadrant, we may suppose that $0<\alpha_1<\alpha_2<\cdots<\alpha_s$. From this we know that the Hessian of $S$ vanishes on some lines crossing the origin. It also reveals an obvious fact
\begin{equation*}
\sum_{l=1}^sm_l+2r=n-2.
\end{equation*}
It should be noted that there is an implicit variety $y=x$ because of the singular kernel $K(x,y)$.
Before proceeding further, let us introduce some notations. Assume that $\mathcal{K}$ is a positive constant depending on $\alpha_1,\cdots,\alpha_s$ thus on $S$. Let $j\gg k$($j\ll k$) represent $j>k+\mathcal{K}$($j<k-\mathcal{K}$) such that the size of $y$-variable($x$-variable) is dominant in the Hessian $S_{xy}^{''}$, while $j\sim k$ naturally means $|j-k|\leq \mathcal{K}$.
To make full use of the nondegeneracy of the Hessian when localized on dyadic areas, we divide $T_2$ into three groups as follows.
\begin{align*}
T_2f(x)&=\sum_{j\gg k}T_{j,k}f(x)+\sum_{j\sim k}T_{j,k}f(x)+\sum_{j\ll k}T_{j,k}f(x)\\
&=T_Yf(x)+T_\Delta f(x)+T_Xf(x).
\end{align*}
If we can establish \eqref{Main-Inq} for $T_X, T_\Delta$ and $T_Y$ individually, then the proof of our main result is complete.

\section{Damped operators}
This section is devoted to proving some damped estimates which will be used to deal with $T_X$ and $T_Y$. In fact, our main strategy is to insert $T_X$ or $T_Y$ into the following two families of operators
\begin{align}
\notag T_Y^zf(x)&=\sum_{j\gg k}\int_{\R}e^{i\lambda S(x,y)}K(x,y)\abs{S^{''}_{xy}}^{z}\left[1-\phi\left((x-y)\lambda^{\frac{1}{n}}\right)\right]\cdot\\
\notag&\quad\quad\quad \phi_j(x)\phi_k(y)\psi(x,y)f(y)dy,\\
&:=\sum_{j\gg k}D^Y_{j,k}f(x),\label{T_X}\\
\notag T_X^zf(x)&=\sum_{j\ll k}\int_{\R}e^{i\lambda S(x,y)}K(x,y)\abs{S^{''}_{xy}}^{z}\left[1-\phi\left((x-y)\lambda^{\frac{1}{n}}\right)\right]\cdot\\
\notag &\quad\quad\quad\phi_j(x)\phi_k(y)\psi(x,y)f(y)dy\\
&:=\sum_{j\ll k}D^X_{j,k}f(x).\label{T_Y}
\end{align}
For each operator we establish the sharp $L^2$ decay estimate as well as the endpoint estimate, by adequate interpolations we get the desired results. This idea first appeared in \cite{shi2018uniform} and later was used in \cite{shi2019damping} to give a new proof of sharp $L^p$ decay of real-analytic oscillatory integral operators.
The sharp $L^2$ decay estimates for $T_X^z$ and $T_Y^z$ state as follows.
\begin{theorem}\label{dampL2}
If the Hessian of the phase function $S(x,y)$ is of the form \eqref{Hessian}, then for $\Re(z)=\frac{1}{2}$ we have
\begin{align}
&\norm{T_Y^zf}_{L^2}\leq C_{E,\mathcal{K},\psi,n,z}\lambda^{\frac{\mu}{n}-\frac{1}{2}}\norm{f}_{L^2},\label{dampL2-1}\\
&\norm{T_X^zf}_{L^2}\leq C_{E,\mathcal{K},\psi,n,z}\lambda^{\frac{\mu}{n}-\frac{1}{2}}\norm{f}_{L^2}.\label{dampL2-2}
\end{align}
\end{theorem}
To establish these two estimates, we need to give a local oscillatory estimate. Consider the following damped operator
 \begin{equation}\label{DampO}
D(\mathcal{B})f(x)=\int_{\R} e^{i\lambda S(x,y)}\abs{S^{''}_{xy}}^{\frac{1}{2}}K(x,y)\left[1-\phi\left((x-y)\lambda^{\frac{1}{n}}\right)\right]\psi(x,y)f(y)dy,
\end{equation}
 where $K(x,y)$ is defined in \eqref{SinCond}, $\psi\in C_0^{\infty}$ and $\supp\psi\subset \mathcal{B}$. Now we consider two operators $D(\mathcal{B}_1)$ and $D(\mathcal{B}_2)$ with supports $\mathcal{B}_1$ and $\mathcal{B}_2$ respectively. Here both $\mathcal{B}_1$ and $\mathcal{B}_2$ are rectangular boxes with sides parallel to the axes; in addition, the minor box $\mathcal{B}_2$ will be contained in a horizontal translate of the major box $\mathcal{B}_1$. As demonstrated in \cite{phong1998damped}, we give the precise definitions and assumptions.
\begin{align*}
&\mathcal{B}_1=\{(x,y): a_1<x<b_1, c_1<y<d_1\}, \rho_1=d_1-c_1;\\
&\widetilde{\mathcal{B}_1}=\left\{(x,y): a_1-\frac{1}{10}(b_1-a_1)<x<b_1+\frac{1}{10}(b_1-a_1), c_1<y<d_1\right\};\\
&\mathcal{B}_1^*=\{(x,y): a_1-(b_1-a_1)<x<b_1+(b_1-a_1), c_1<y<d_1\}.
\end{align*}
We also have the minor box $\mathcal{B}_2$
\begin{equation*}
\mathcal{B}_2=\{(x,y): a_2<x<b_2, c_2<y<d_2\}, \rho_2=d_2-c_2.
\end{equation*}
For these two boxes, we have the following assumptions:
\begin{enumerate}
\item[(A1)] We define the span \emph{span}($\mathcal{B}_1,\mathcal{B}_2$), as the union of all line segments parallel to the $x$-axis, which joints a point $(x,y)\in\mathcal{B}_1$ with a point $(z,y)\in\mathcal{B}_2$. While we also assume that $S_{xy}^{''}$ does not change sign in the span \emph{span}($\mathcal{B}_1,\mathcal{B}_2$) and satisfies
    \begin{align}
    &\nu\leq \min_{\widetilde{\mathcal{B}_1}}\abs{S_{xy}^{''}}\leq A\nu,\\
    &\max_{\emph{span}(\mathcal{B}_1,\mathcal{B}_2)}\abs{S_{xy}^{''}}\leq A\nu.
    \end{align}
\item[(A2)] $\mathcal{B}_2\subset \mathcal{B}_1^*$, this implies $\rho_2\leq \rho_1$.
\end{enumerate}
 For the cut-off functions $\psi_j(x,y)$, we also assume that
\begin{enumerate}
\item[(A3)] $\sum_k\rho_j^{k}\abs{\partial_y^{k}\psi_j}\leq B$.
\end{enumerate}
 Now we formulate an almost-orthogonality principle which we will rely on to establish Theorem \ref{dampL2}.
\begin{lemma}\label{Orth_2}
Under the assumptions (A1)-(A3), we have
\begin{align}
\norm{D(\mathcal{B}_1)D(\mathcal{B}_2)^{*}}_{L^2}\leq C_{E,A,B,\mu,n} \lambda^{\frac{2\mu}{n}-1}\frac{\sup_{\mathcal{B}_2}\abs{S_{xy}^{''}}^{\frac{1}{2}}}{\sup_{\widetilde{\mathcal{B}_1}}\abs{S_{xy}^{''}}^{\frac{1}{2}}},\label{Orth_2-1}\\
\norm{D(\mathcal{B}_2)D(\mathcal{B}_1)^{*}}_{L^2}\leq C_{E,A,B,\mu,n} \lambda^{\frac{2\mu}{n}-1}\frac{\sup_{\mathcal{B}_2}\abs{S_{xy}^{''}}^{\frac{1}{2}}}{\sup_{\widetilde{\mathcal{B}_1}}\abs{S_{xy}^{''}}^{\frac{1}{2}}}\label{Orth_2-2}
\end{align}
where the constant $C_{A,B,n}$ depends only on $A$, $B$ and $n$.
\end{lemma}

\begin{remark}\label{OpVanDerCorput}
In Lemma \ref{Orth_2}, if we set $\mathcal{B}_1=\mathcal{B}_2$ which we denote by $\mathcal{B}$, we have
\begin{equation*}
\norm{D(\mathcal{B})}_{L^2}\leq C_{E,A,B,\mu,n}\lambda^{\frac{\mu}{n}-\frac{1}{2}}.
\end{equation*}
This corresponds to Lemma 1 in \cite{liu1999model}.
\end{remark}
\begin{remark}\label{DualLemma}
Interchanging the roles of $x$ and $y$ in assumptions (A1)-(A3), \eqref{Orth_2-1} and \eqref{Orth_2-2} also hold for operators $D(\mathcal{B}_1)^{*}D(\mathcal{B}_2)$ and $D(\mathcal{B}_2)^{*}D(\mathcal{B}_1)$ respectively.
\end{remark}
Before the formal proof of Lemma Lemma \ref{Orth_2}, we first show how Lemma \ref{Orth_2} implies Theorem \ref{dampL2}.
\begin{proof}[Proof that Lemma \ref{Orth_2} implies Theorem \ref{dampL2}]
Recall that
\begin{align*}
T_X^zf(x)&=\sum_{j\ll k}D^X_{j,k}f(x),\\
T_Y^zf(x)&=\sum_{j\gg k}D^Y_{j,k}f(x).
\end{align*}
We rewrite them as
\begin{align}
T_X^zf(x)&=\sum_{j\ll k}D^X_{j,k}f(x):=\sum_{j}D^X_{j}f(x),\label{DX}\\
T_Y^zf(x)&=\sum_{j\gg k}D^Y_{j,k}f(x):=\sum_{k}D^Y_{k}f(x).\label{DY}
\end{align}
Evidently, the amplitude of $D^X_{j}$ is supported in a rectangle of the form $\{x\sim 2^{-j}\}\times\{0<y<2^{-j}\}$.
For $j'\neq j$, from Lemma \ref{Orth_2} we know that
\begin{align*}
\norm{D^X_{j}\left(D^X_{j'}\right)^{*}}_{L^2}\leq C_{E,\mathcal{K},\psi,\mu,n} \lambda ^{\frac{2\mu}{n}-1}2^{-\frac{\abs{j-j'}(n-2)}{2}},\\
\norm{D^X_{j'}\left(D^X_{j}\right)^{*}}_{L^2}\leq C_{E,\mathcal{K},\psi,\mu,n} \lambda ^{\frac{2\mu}{n}-1}2^{-\frac{\abs{j-j'}(n-2)}{2}}.
\end{align*}
For $k\neq k'$, from Remark \ref{DualLemma}, we also have
\begin{align*}
\norm{\left(D^Y_{k}\right)^{*}D^Y_{k'}}_{L^2}\leq C_{E,\mathcal{K},\psi,\mu,n} \lambda ^{\frac{2\mu}{n}-1}2^{-\frac{\abs{k-k'}(n-2)}{2}},\\
\norm{\left(D^Y_{k'}\right)^{*}D^Y_{k}}_{L^2}\leq C_{E,\mathcal{K},\psi,\mu,n} \lambda ^{\frac{2\mu}{n}-1}2^{-\frac{\abs{k-k'}(n-2)}{2}}.
\end{align*}
On the other hand, $\{D^X_{j}\}$ and $\{D^Y_{k}\}$ are two sequences of operators that are pairwise essentially disjoint in $x-$variable and $y-$variable respectively. By Cotlar-Stein Lemma, we can conclude Theorem \ref{dampL2} and finish the proof.
\end{proof}
Now we return to the proof of Lemma \ref{Orth_2}.
\begin{proof}
First, we compute the kernel of $D(\mathcal{B}_1)D(\mathcal{B}_2)^{*}$ as follows.
\begin{align*}
Ker(D(\mathcal{B}_1)D(\mathcal{B}_2)^{*})(x,z)=\int& e^{i\lambda\left[S(x,y)-S(z,y)\right]}\left[1-\phi\left((x-y)\lambda^{\frac{1}{n}}\right)\right]\left[1-\phi\left((z-y)\lambda^{\frac{1}{n}}\right)\right]\cdot\\
&K(x,y)\overline{K(z,y)}\abs{S^{''}_{xy}}^{\frac{1}{2}}\abs{S^{''}_{zy}}^{\frac{1}{2}}\psi_1(x,y)\overline{\psi_2(z,y)}dy,
\end{align*}
where $\supp(\psi_1)\subset \mathcal{B}_1, \supp(\psi_2)\subset \mathcal{B}_2$.
Set
\begin{align*}
\Phi(x,y,z)&=S^{'}_{y}(x,y)-S^{'}_{y}(z,y)=\int_z^x S^{''}_{uy}(u,y)du,\\
\Lambda_1(x,y,z)&=\left[1-\phi\left((x-y)\lambda^{\frac{1}{n}}\right)\right]\left[1-\phi\left((z-y)\lambda^{\frac{1}{n}}\right)\right]K(x,y)\overline{K(z,y)},\\
\Lambda_2(x,y,z)&=\abs{S^{''}_{xy}}^{\frac{1}{2}}\abs{S^{''}_{zy}}^{\frac{1}{2}}\psi_1(x,y)\overline{\psi_2(z,y)}.
\end{align*}
For fixed $y$, if both $(x,y)$ and $(z,y)$ are contained in $\widetilde{\mathcal{B}_1}$, from assumption (A1), we know that
\begin{equation}
\abs{x-z}\leq \abs{\Phi(x,y,z)}\leq A\nu\abs{x-z} ,\label{Poly-Property1}\\
\end{equation}
In view of the fact that $S(x,y)$ is a polynomial of degree at most $n$, from Lemma 1.2 in \cite{phong1994models}, we also have
\begin{equation}
\abs{\partial^k_y\Phi(x,y,z)}\leq C_{A,n}\nu\abs{x-z}\rho_1^{-k}.\label{Poly-Property2}
\end{equation}
For fixed $y$, if $(z,y)$ is outside $\widetilde{\mathcal{B}_1}$, we can see that
\begin{equation*}
\abs{\Phi(x,y,z)}\geq \abs{\int_{\tilde{x}}^x S^{''}_{uy}(u,y)du},
\end{equation*}
where $\tilde{x}$ is between $x$ and $z$, in addition, $(\tilde{x}, y)$ is also on the edge of $\widetilde{\mathcal{B}_1}$. Obviously,
\begin{equation*}
\abs{\Phi(x,y,z)}\geq \nu\abs{x-\tilde{x}}.
\end{equation*}
Recall the definition of $\widetilde{\mathcal{B}_1}$, since $\mathcal{B}_2\subset\mathcal{B}_1^{*}$, then
\begin{equation}\label{Poly-Property3}
A\nu\abs{x-z}\geq\abs{\Phi(x,y,z)}\geq \nu\abs{x-\tilde{x}}\geq \nu\cdot\frac{b_1-a_1}{10}\geq \nu\cdot\frac{\abs{x-z}}{30}.
\end{equation}
Similarly, in this case,
\begin{equation}\label{Poly-Property4}
\abs{\partial^k_y\Phi(x,y,z)}\leq C_{A,n}\nu\abs{x-z}\rho_2^{-k}.
\end{equation}

On one hand, we have the trivial size estimate
\begin{align}
\abs{Ker(D(\mathcal{B}_1)D(\mathcal{B}_2)^{*})(x,z)}&\leq C_{E,B,\mu}\lambda^{\frac{2\mu}{n}}\rho_2\sup_{\mathcal{B}_1}\abs{S_{xy}^{''}}^{\frac{1}{2}}\sup_{\mathcal{B}_2}\abs{S_{xy}^{''}}^{\frac{1}{2}}\notag\\
&\leq C_{E,A,B,\mu}\nu\rho_2\frac{\sup_{\mathcal{B}_2}\abs{S_{zy}^{''}}^{\frac{1}{2}}}{\sup_{\widetilde{\mathcal{B}_1}}\abs{S_{xy}^{''}}^{\frac{1}{2}}}.\label{S-Estimate}
\end{align}
On the other hand, we also claim that
\begin{align}\label{O-Estimate}
\abs{Ker(D(\mathcal{B}_1)D(\mathcal{B}_2)^{*})(x,z)}&\leq C_{A,B,n} \frac{\lambda^{\frac{2\mu}{n}}\rho_2^{-1}\sup_{\mathcal{B}_1}\abs{S_{xy}^{''}}^{\frac{1}{2}}\sup_{\mathcal{B}_2}\abs{S_{xy}^{''}}^{\frac{1}{2}}}{\lambda^2\nu^2\abs{x-z}^2}\\
&\leq C_{A,B,n} \frac{\lambda^{\frac{2\mu}{n}}\rho_2^{-1}\nu}{\lambda^2\nu^2\abs{x-z}^2}\cdot\frac{\sup_{\mathcal{B}_2}\abs{S_{zy}^{''}}^{\frac{1}{2}}}{\sup_{\widetilde{\mathcal{B}_1}}\abs{S_{xy}^{''}}^{\frac{1}{2}}}.
\end{align}
On account of the assumptions, the second inequality is easy. For the first inequality, integration by parts yields
\begin{align*}
Ker(D(\mathcal{B}_1)D(\mathcal{B}_2)^{*})(x,z)=\frac{1}{i\lambda}\int& e^{i\lambda\left[S(x,y)-S(z,y)\right]}\frac{d}{dy}\left(\frac{\Lambda_1\Lambda_2}{\Phi}\right)dy\\
=-\frac{1}{\lambda^2}\int& e^{i\lambda\left[S(x,y)-S(z,y)\right]}\frac{d}{dy}\left[\frac{1}{\Phi}\frac{d}{dy}\left(\frac{\Lambda_1\Lambda_2}{\Phi}\right)\right]dy\\
=-\frac{1}{\lambda^2}\int& e^{i\lambda\left[S(x,y)-S(z,y)\right]}\frac{d}{dy}\left[\frac{d\Lambda_1}{dy}\cdot\frac{\Lambda_2}{\Phi^2}+\frac{\Lambda_1}{\Phi}\cdot\frac{d}{dy}
\left(\frac{\Lambda_2}{\Phi}\right)\right]dy\\
=-\frac{1}{\lambda^2}\int& e^{i\lambda\left[S(x,y)-S(z,y)\right]}\frac{d\Lambda_1}{dy}\cdot\left[\frac{1}{\Phi}\cdot\frac{d}{dy}\left(\frac{\Lambda_2}{\Phi}\right)+\frac{d}{dy}
\left(\frac{\Lambda_2}{\Phi^2}\right)\right]+\\
&\frac{d}{dy}\left[\frac{1}{\Phi}\cdot\frac{d}{dy}\left(\frac{\Lambda_2}{\Phi}\right)\right]\cdot\Lambda_1+\frac{\Lambda_2}{\Phi^2}\cdot\frac{d^2\Lambda_1}{dy^2}dy.
\end{align*}
From \cite{liu1999model}, we know that
\begin{equation*}
\abs{\Lambda_1}\leq C_{E,\mu}\lambda^{\frac{2\mu}{n}},\quad\int\abs{\frac{d\Lambda_1}{dy}}dy\leq C_{E,\mu}\lambda^{\frac{2\mu}{n}},\quad \int\abs{\frac{d^2\Lambda_1}{dy^2}}dy\leq C_{E,\mu}\lambda^{\frac{2\mu}{n}}\rho_2^{-1}.
\end{equation*}
The other terms in the integrand can be dealt with using \eqref{Poly-Property1}-\eqref{Poly-Property4}, details can also be found in \cite{phong1994models}, here we only list the facts
\begin{align*}
&\abs{\frac{1}{\Phi}\cdot\frac{d}{dy}\left(\frac{\Lambda_2}{\Phi}\right)+\frac{d}{dy}
\left(\frac{\Lambda_2}{\Phi^2}\right)}\leq C_{A,B,n}\frac{\rho_2^{-1}\sup_{\mathcal{B}_1}\abs{S_{xy}^{''}}^{\frac{1}{2}}\sup_{\mathcal{B}_2}\abs{S_{xy}^{''}}^{\frac{1}{2}}}{\nu^2\abs{x-z}^2},\\
&\abs{\frac{d}{dy}\left[\frac{1}{\Phi}\cdot\frac{d}{dy}\left(\frac{\Lambda_2}{\Phi}\right)\right]}\leq C_{A,B,n}\frac{\rho_2^{-2}\sup_{\mathcal{B}_1}\abs{S_{xy}^{''}}^{\frac{1}{2}}\sup_{\mathcal{B}_2}\abs{S_{xy}^{''}}^{\frac{1}{2}}}{\nu^2\abs{x-z}^2},\\
&\abs{\frac{\Lambda_2}{\Phi^2}}\leq  C_{A,B}\frac{\sup_{\mathcal{B}_1}\abs{S_{xy}^{''}}^{\frac{1}{2}}\sup_{\mathcal{B}_2}\abs{S_{xy}^{''}}^{\frac{1}{2}}}{\nu^2\abs{x-z}^2}.
\end{align*}
Collecting all estimates above, we can conclude \eqref{O-Estimate}. Also, in view of \eqref{S-Estimate}, we obtain
\begin{equation*}
\abs{Ker(D(\mathcal{B}_1)D(\mathcal{B}_2)^{*})(x,z)}\leq C_{E,A,B,\mu,n} \lambda^{\frac{2\mu}{n}-1}\frac{\lambda\nu\rho_2}{1+\lambda^2\nu^2\rho_2^2\abs{x-z}^2}\cdot\frac{\sup_{\mathcal{B}_2}\abs{S_{zy}^{''}}^{\frac{1}{2}}}{\sup_{\widetilde{\mathcal{B}_1}}\abs{S_{xy}^{''}}^{\frac{1}{2}}}.
\end{equation*}
It follows that
\begin{align*}
\abs{D(\mathcal{B}_1)D(\mathcal{B}_2)^{*}f(x)}&\leq\int\abs{Ker(DD^*)(x,z)}\abs{f(z)}dz\\
&\leq  C_{E,A,B,\mu,n}\lambda^{\frac{2\mu}{n}-1}\int \frac{\lambda\nu\rho_2}{1+\lambda^2\nu^2\rho_2^2\abs{x-z}^2}\abs{f(z)}dz\\
&\leq C_{E,A,B,\mu,n}\lambda^{\frac{2\mu}{n}-1}\frac{\sup_{\mathcal{B}_2}\abs{S_{zy}^{''}}^{\frac{1}{2}}}{\sup_{\widetilde{\mathcal{B}_1}}\abs{S_{xy}^{''}}^{\frac{1}{2}}}Mf(x),
\end{align*}
here $M$ is the Hardy-Littlewood maximal operator. Due to the $L^2$ boundedness of the Hardy-Littlewood maximal operator, we can conclude that
\begin{equation*}
\norm{D(\mathcal{B}_1)D(\mathcal{B}_2)^{*}f}_{L^2}\leq C_{E,A,B,\mu,n}\lambda^{\frac{2\mu}{n}-1}\frac{\sup_{\mathcal{B}_2}\abs{S_{zy}^{''}}^{\frac{1}{2}}}{\sup_{\widetilde{\mathcal{B}_1}}\abs{S_{xy}^{''}}^{\frac{1}{2}}}\norm{f}_{L^2}.
\end{equation*}
This gives \eqref{Orth_2-1}. The inequality \eqref{Orth_2-2} follows by taking adjoints. Thus we complete the proof of Theorem \ref{dampL2}.\\
\end{proof}
To get the $L^p$ estimate we also need the following endpoint estimates.
\begin{theorem}\label{EndPo1}
With the same requirements as Theorem \ref{dampL2} for the phase function, then for $\Re(z)=-\frac{1-\mu}{n-2}$ we have
\begin{align}
&\norm{T_X^zf}_{L^{1,\infty}}\leq C_{E,\psi}\norm{f}_{L^1},\label{EndPo-1}\\
&\norm{T_Y^zf}_{L^1}\leq C_{E,\psi}\norm{f}_{L^1}.\label{EndPo-2}
\end{align}
\end{theorem}
\begin{proof}
Observe that
\begin{align*}
\abs{T_X^zf(x)}&\leq C_{\psi}\int_{\abs{y}\leq \abs{x}}\abs{K(x,y)}\abs{S^{''}_{xy}}^{-\frac{1-\mu}{n-2}}\abs{f(y)}dy\\
&\leq C_{E,\psi}\int_{\abs{y}\leq \abs{x}}\abs{x}^{-\mu}\abs{x}^{-1+\mu}\abs{f(y)}dy\\
&=C_{E,\psi}\abs{x}^{-1}\int_{\abs{y}\leq\abs{x}}\abs{f(y)}dy\\
&\leq C_{E,\psi}\abs{x}^{-1}\norm{f}_{L^1}.
\end{align*}
If $f\in L^{1}$, we can easily conclude \eqref{EndPo-1}. Now we turn to prove \eqref{EndPo-2}. Similarly,
\begin{align*}
\abs{T_Y^zf(x)}&\leq C_{\psi}\int_{\abs{y}\geq\abs{x}}\abs{K(x,y)}\abs{S^{''}_{xy}}^{-\frac{1-\mu}{n-2}}\abs{f(y)}dy\\
&\leq C_{E,\psi}\int_{\abs{y}\geq\abs{x}}\abs{y}^{-\mu}\abs{y}^{-1+\mu}\abs{f(y)}dy\\
&=C_{E,\psi}\int_{\abs{y}\geq\abs{x}}\abs{y}^{-1}\abs{f(y)}dy.
\end{align*}
By Fubini's theorem, we have
\begin{align*}
\norm{T_Y^zf}_{L^1}&\leq C_{E,\psi}\iint_{\abs{y}\geq\abs{x}}\abs{y}^{-1}\abs{f(y)}dydx\\
&=C_{E,\psi}\int\abs{y}^{-1}\abs{f(y)}\int_{\abs{y}\geq\abs{x}}dxdy\\
&\leq C_{E,\psi}\int\abs{y}^{-1}\abs{y}\abs{f(y)}dy\\
&=C_{E,\psi}\norm{f}_{L^1}.
\end{align*}
This implies \eqref{EndPo-2}.
\end{proof}
For the sake of interpolation, we also need the following lemma with change of power weights. An earlier version of this lemma appeared in \cite{pansamsze1997}, see also \cite{shi2017sharp},\cite{shi2018uniform} for details of proof.
\begin{lemma}\label{InterpolationLem}
Let $dx$ be the Lebesgue measure on $\R$. Assume $V$ is a linear operator defined on all simple functions with respect to $dx$. If there exist two constant $A_1, A_2>0$ such that
\begin{enumerate}
\item $\norm{Vf}_{L^\infty(dx)}\leq A_1\norm{f}_{L^1(dx)}$ for all simple functions $f$,
\item $\norm{\abs{x}^aVf}_{L^{p_0}(dx)}\leq A_2\norm{f}_{L^{p_0}(dx)}$ for some $1<p_0, a\in \R$ satisfying $ap_0\neq -1$,
\end{enumerate}
then for any $\theta\in (0,1)$, there exists a constant $C=C(a,p_0,\theta)$ such that
\begin{equation}
\norm{\abs{x}^bVf}_{L^{p}(dx)}\leq CA_1^{\theta}A_2^{1-\theta}\norm{f}_{L^{p}(dx)}
\end{equation}
for all simple function $f$, where $b$ and $p$ satisfy $b=-\theta+(1-\theta)a$ and $\frac{1}{p}=\theta+\frac{1-\theta}{p_0}$.
\end{lemma}

\section{Local Riesz-Thorin interpolation}
Now, we will establish the sharp $L^p$ decay estimates for $T_X$ and $T_Y$ whenever $\frac{n-2\mu}{n-1-\mu}< p< \frac{n-2\mu}{1-\mu}$. It reads
\begin{align}
\|T_Xf\|_{L^p}&\leq C_{E,\psi}\lambda^{-\frac{1-\mu}{n}}\|f\|_{L^p},\quad \quad \frac{n-2\mu}{n-1-\mu}<p<+\infty\label{T_XImed}\\
\|T_Yf\|_{L^p}&\leq C_{E,\psi}\lambda^{-\frac{1-\mu}{n}}\|f\|_{L^p}, \quad \quad 1<p< \frac{n-2\mu}{1-\mu}.\label{T_YImed}
\end{align}
In fact, the main strategy, the local interpolation with trivial endpoint estimates, was essentially introduced in \cite{xiao2017endpoint}. In this paper, we modify it to adapt to our model.
\begin{lemma}
Consider the singular oscillatory integral operator
\begin{equation*}
\widetilde{D}f(x)=\int_{\R} e^{i\lambda S(x,y)}K(x,y)\left[1-\phi\left((x-y)\lambda^{\frac{1}{n}}\right)\right]\psi(x,y)f(y)dy,
\end{equation*}
with the same assumptions of Lemma 1 in \cite{liu1999model}, furthermore, we also denote the upper bounds of $x-$cross section and $y-$cross section of the support of $\psi(x,y)$ by $\delta_1$ and $\delta_2$ respectively. If we denote the kernel of $\widetilde{D}$ by $K_\lambda(x,y)$ and $\abs{K_\lambda(x,y)}\leq C_K$, then we have
\begin{align}
\norm{\widetilde{D}}_{L^p}&\leq C_{E,\psi,n,p}\min\left\{\lambda^{\left(\frac{2\mu}{n}-1\right)\cdot\frac{1}{p'}}\nu^{-\frac{1}{p'}}C_{K}^{\frac{2}{p}-1}\delta_1^{\frac{2}{p}-1},\quad C_K\delta_1^{\frac{1}{p}}\delta_2^{\frac{1}{p'}}\right\},\quad \quad 1<p<2;\label{DLp<2}\\
\norm{\widetilde{D}}_{L^p}&\leq C_{E,\psi,n,p}\min\left\{\lambda^{\left(\frac{2\mu}{n}-1\right)\cdot\frac{1}{p}}\nu^{-\frac{1}{p}}C_{K}^{1-\frac{2}{p}}\delta_2^{1-\frac{2}{p}}, \quad C_K\delta_1^{\frac{1}{p}}\delta_2^{\frac{1}{p'}}\right\},\ \quad \quad 2<p<+\infty.\label{DLp>2}
\end{align}
\end{lemma}
This can be deduced by interpolating the $L^2$ estimates, shown in \cite{liu1999model},
\begin{align*}
\norm{\widetilde{D}}_{L^2}&\leq C_{E,\psi,n}\lambda^{\frac{\mu}{n}-\frac{1}{2}}\nu^{-\frac{1}{2}},\\
\norm{\widetilde{D}}_{L^2}&\leq C_K\left(\delta_1\delta_2\right)^{\frac{1}{2}},
\end{align*}
with the endpoint estimates
\begin{align*}
\norm{\widetilde{D}}_{L^\infty}&\leq C_K\delta_2,\\
\norm{\widetilde{D}}_{L^1}&\leq C_K\delta_1.
\end{align*}
Now we turn to prove \eqref{T_XImed} and \eqref{T_YImed}.
\begin{proof}
For $j\gg k$, we can see
\begin{align}
\notag&\nu=C_{\mathcal{K}}2^{-k(n-2)}, &\delta_1\approx2^{-j};\\
&C_K=C_{E,\mathcal{K},\psi}2^{\mu k},&\delta_2\approx2^{-k}.\label{Fact1}
\end{align}
If $p>2$, by invoking \eqref{DLp>2}, we can deduce that
\begin{align*}
\norm{T_{j,k}}_{L^p}\leq C_{E,\mathcal{K},\psi,n,p}\min\left\{\lambda^{\left(\frac{2\mu}{n}-1\right)\cdot\frac{1}{p}}\left(2^{-k(n-2)}\right)^{-\frac{1}{p}}\left(2^{\mu k}\right)^{1-\frac{2}{p}}\left(2^{-k}\right)^{1-\frac{2}{p}}, \quad 2^{\mu k}2^{-\frac{j}{p}}2^{-\frac{k}{p'}}\right\}.
\end{align*}
Since $j\gg k$, we may set $j=k+M$, then from above we have
\begin{align*}
\norm{T_{j,k}}_{L^p}\leq C_{E,\mathcal{K},\psi,n,p}\min\left\{\lambda^{\left(\frac{2\mu}{n}-1\right)\cdot\frac{1}{p}}\left(2^k\right)^{\frac{n}{p}-1+\mu\left(1-\frac{2}{p}\right)}, \quad 2^{(\mu-1) k}2^{-\frac{M}{p}}\right\}.
\end{align*}
Therefore,
\begin{align*}
\norm{T_Y}_{L^p}&=\norm{\sum_{j\gg k}T_{j,k}}_{L^p}\\
&\leq C_{E,\mathcal{K},\psi,n,p}\sum_{M=0}^{+\infty}\sum_{k=0}^{+\infty}\min\left\{\lambda^{\left(\frac{2\mu}{n}-1\right)\cdot\frac{1}{p}}\left(2^k\right)^{\frac{n}{p}-1+\mu\left(1-\frac{2}{p}\right)}, \quad 2^{(\mu-1) k}2^{-\frac{M}{p}}\right\}.
\end{align*}
We derive from
\begin{equation*}
\lambda^{\left(\frac{2\mu}{n}-1\right)\cdot\frac{1}{p}}\left(2^k\right)^{\frac{n}{p}-1+\mu\left(1-\frac{2}{p}\right)}\approx 2^{(\mu-1) k}2^{-\frac{M}{p}}
\end{equation*}
that
\begin{equation*}
2^{k}\approx \lambda^{\frac{1}{n}}2^{\frac{M}{2\mu-n}}.
\end{equation*}
It implies
\begin{align*}
\sum_{k=0}^{+\infty}\min\left\{\lambda^{\left(\frac{2\mu}{n}-1\right)\cdot\frac{1}{p}}\left(2^k\right)^{\frac{n}{p}-1+\mu\left(1-\frac{2}{p}\right)}, \quad 2^{(\mu-1) k}2^{-\frac{M}{p}}\right\}&\approx 2^{-\frac{M}{p}}\left(\lambda^{\frac{1}{n}}2^{\frac{M}{2\mu-n}}\right)^{\mu-1}\\
&=\lambda^{\frac{\mu-1}{n}}2^{\left(\frac{\mu-1}{2\mu-n}-\frac{1}{p}\right)\cdot M}.
\end{align*}
To make the sum above converge, we require that
\begin{equation*}
\frac{\mu-1}{2\mu-n}-\frac{1}{p}<0,
\end{equation*}
which equals
\begin{equation*}
p<\frac{n-2\mu}{1-\mu}.
\end{equation*}
It should be noted that $p>2$ determines $n>2$. Actually, the case $n=2$ has been discussed in \cite{liu1999model}. At present, we continue to discuss the $L^p$ estimate of $T_Y$ whenever $1<p<2$.  On account of \eqref{DLp<2} and \eqref{Fact1}, we can see
\begin{align*}
\norm{T_{j,k}}_{L^p}&\leq C_{E,\mathcal{K},\psi,n,p}\min\left\{\lambda^{\left(\frac{2\mu}{n}-1\right)\cdot\frac{1}{p'}}\left(2^{-k(n-2)}\right)^{-\frac{1}{p'}}\left(2^{\mu k}\right)^{\frac{2}{p}-1}\left(2^{-j}\right)^{\frac{2}{p}-1}, \quad 2^{\mu k}2^{-\frac{j}{p}}2^{-\frac{k}{p'}}\right\}\\
&=C_{E,\mathcal{K},\psi,n,p}\min\left\{\lambda^{\left(\frac{2\mu}{n}-1\right)\cdot\frac{1}{p'}}\left(2^k\right)^{\frac{n}{p'}-1+\mu\left(\frac{2}{p}-1\right)}2^{-M\left(\frac{2}{p}-1\right)}, \quad 2^{(\mu-1) k}2^{-\frac{M}{p}}\right\}.
\end{align*}
Similarly, if
\begin{equation*}
\lambda^{\left(\frac{2\mu}{n}-1\right)\cdot\frac{1}{p'}}\left(2^k\right)^{\frac{n}{p'}-1+\mu\left(\frac{2}{p}-1\right)}2^{-M\left(\frac{2}{p}-1\right)}\approx 2^{(\mu-1) k}2^{-\frac{M}{p}},
\end{equation*}
then
\begin{equation*}
2^{k}\approx \lambda^{\frac{1}{n}}2^{\frac{M}{2\mu-n}}.
\end{equation*}
Therefore
\begin{align*}
\norm{T_Y}_{L^p}&=\norm{\sum_{j\gg k}T_{j,k}}_{L^p}\\
&\leq C_{E,\mathcal{K},\psi,n,p}\sum_{M=0}^{+\infty}\sum_{k=0}^{+\infty}\min\left\{\lambda^{\left(\frac{2\mu}{n}-1\right)\cdot\frac{1}{p'}}\left(2^k\right)^{\frac{n}{p'}-1+\mu\left(\frac{2}{p}-1\right)}2^{-M\left(\frac{2}{p}-1\right)}, \quad 2^{(\mu-1) k}2^{-\frac{M}{p}}\right\}\\
&\leq C_{E,\mathcal{K},\psi,n,p}\sum_{M=0}^{+\infty}\lambda^{\frac{\mu-1}{n}}2^{\left(\frac{\mu-1}{2\mu-n}-\frac{1}{p}\right)\cdot M}.
\end{align*}
Taking $1<p<2$ into consideration, we know that
\begin{equation*}
\frac{\mu-1}{2\mu-n}-\frac{1}{p}<0.
\end{equation*}
The sum converges, thus we complete the proof of \eqref{T_YImed}. Repeating the argument above for $T_X$, we can get \eqref{T_XImed} similarly. The proof is finished.
\end{proof}
\section{The remaining case: $T_\Delta$}
This section is devoted to establishing the sharp $L^p$ decay estimates for $T_\Delta$. Preceding the formal proof, we state a useful almost-orthogonality principle which was introduced in \cite{phong2001multilinear}.
\begin{lemma}
For the bilinear operator $V(f,g)=\int_{\Omega_1\times\Omega_2}K(x,y)f(x)g(y)dxdy$, assume that
\begin{equation*}
\left\{(x,y): K(x,y)\neq 0\right\}\subset\bigcup_{k=1}^{\infty}\left(I_k\times J_k\right),
\end{equation*}
where $\{I_k\}$ and $\{J_k\}$ are two groups of mutually disjoint measurable subsets of $\Omega_1$ and $\Omega_2$. Let $V_k$ be the bilinear operator with kernel $\chi_{I_k}(x)\chi_{J_k}(y)K(x,y)$, let $\norm{V_k}$ and $\norm{V}$ be the norm of the $T_k$ and $T$ as bilinear operators on $L^p(\Omega_1)\times L^{p'}(\Omega_2)$. Then
\begin{equation}\label{orth}
\norm{V}\leq \sup_k\norm{V_k}.
\end{equation}
\end{lemma}
\begin{remark}\label{Orth1}
In fact, if both $\{I_k\}$ and $\{J_k\}$ are groups of sets having finite overlaps, we also have
\begin{equation*}
\norm{V}\leq C\sup_k\norm{V_k},
\end{equation*}
where the implicit constant depends only on the overlapping numbers.
\end{remark}
Now we turn to deal with $T_\Delta$. We claim that
\begin{align}
\|T_\Delta f\|_{L^p}&\leq C_{E,\mathcal{K},\psi,n,p}\lambda^{-\frac{1-\mu}{n}}\|f\|_{L^p},\quad \quad \frac{n-2\mu}{n-1-\mu}\leq p\leq \frac{n-2\mu}{1-\mu}.\label{T_DelEnd}
\end{align}
Note that $T_\Delta f(x)=\sum_{j\sim k}T_{j,k}f(x)$, thus from the remark above, we know that
\begin{equation*}
\norm{T_\Delta}_{L^p}\leq C_{\mathcal{K}}\sup_{j\sim k}\norm{T_{j,k}}_{L^p}.
\end{equation*}
Thus we are reduced to estimating each $T_{j,k}$.
Recalling the Hessian of the phase function, we can list the varieties as follows
\begin{align*}
&y-\alpha_1x=0,&y-\alpha_2x=0, &\quad\quad\quad\quad\cdots, &y-\alpha_sx=0,
\end{align*}
while
\begin{equation*}
0<\alpha_1<\alpha_2<\cdots<\alpha_s.
\end{equation*}
\emph{Case I}: There exists a $l$ such that $\alpha_l=1$.\\
Without loss of generality, we may suppose that $\alpha_1=1$. Along the variety $y-x=0$, we further decompose the support of $T_{j,k}$ dyadically.
\begin{align*}
T_{j,k}f(x)&=\sum_{l_1}\int_{\R} e^{i\lambda S(x,y)}K(x,y)\left[1-\phi\left((x-y)\lambda^{\frac{1}{n}}\right)\right]\phi_{l_1}(y-x)\cdot\\
&\quad\quad\quad\quad\phi_k(x)\phi_k(y)\psi(x,y)f(y)dy\\
&:=\sum_{l_1}T_{j,k}^{l_1}f(x).
\end{align*}
Given that $j\sim k$, in the support of $T_{k,k}^{l_1}$, we know $l_1\geq j\sim k$. So we can divide the sum into two parts as follows
\begin{align}\label{T_Decom}
T_{j,k}f(x)=\sum_{l_1\sim j\sim k}T_{j,k}^{l_1}f(x)+\sum_{l_1\gg j\sim k}T_{j,k}^{l_1}f(x).
\end{align}
To get the final result, it suffices to establish $L^p$ estimates for the above two parts respectively. For the latter one, since $l_1\gg j\sim k$, we may set $l_1=k+M$, then
\begin{align*}
\notag&\nu=C_{\mathcal{K}}2^{-k(n-2)}2^{-M}, &\delta_1\approx2^{-k-M};\\
&C_K=C_{E,\mathcal{K},\psi}2^{\mu(k+M)},&\delta_2\approx2^{-k-M}.
\end{align*}
From \eqref{DLp>2}, we know that for $p>2$ we have
\begin{align*}
\norm{T_{j,k}^{l_1}}_{L^p}&\leq C_{E,\mathcal{K},\psi,n,p}\lambda^{\left(\frac{2\mu}{n}-1\right)\cdot \frac{1}{p}}\left(2^{-k(n-2)-M}\right)^{-\frac{1}{p}}\left(2^{(\mu-1)(k+M)}\right)^{1-\frac{2}{p}},\\
\norm{T_{j,k}^{l_1}}_{L^p}&\leq C_{E,\mathcal{K},\psi,p}2^{(\mu-1)(k+M)}.
\end{align*}
By convex combination, for $0\leq\theta\leq 1$, we have
\begin{align*}
\norm{T_{j,k}^{l_1}}_{L^p}\leq C_{E,\mathcal{K},\psi,n,p}\left[\lambda^{\left(\frac{2\mu}{n}-1\right)\cdot \frac{1}{p}}\left(2^{-k(n-2)-M}\right)^{-\frac{1}{p}}\left(2^{(\mu-1)(k+M)}\right)^{1-\frac{2}{p}}\right]^\theta\cdot\left[2^{(\mu-1)(k+M)}\right]^{1-\theta}.
\end{align*}
We choose suitable $\theta$ to eliminate $k$. This leads to
\begin{align*}
\theta\left[\frac{n-2}{p}+(\mu-1)\left(1-\frac{2}{p}\right)\right]+(\mu-1)(1-\theta)=0.
\end{align*}
Solve this equation about $\theta$ and get the solution
\begin{equation*}
\theta=\frac{p(1-\mu)}{n-2\mu}.
\end{equation*}
The restriction $0\leq\theta\leq 1$ requires $p\leq \frac{n-2\mu}{1-\mu}$. Plugging this into the above convex combination, we obtain
\begin{align*}
 \norm{\sum_{l_1\gg k}T_{j,k}^{l_1}}_{L^p}&\leq C_{E,\mathcal{K},\psi,n,p}\sum_{M=0}^\infty\lambda^{\frac{2\mu-n}{np}\cdot\theta}2^{M\left[\frac{1}{p}+(\mu-1)\left(1-\frac{2}{p}\right)\right]\cdot\theta}2^{(\mu-1)(1-\theta)M}\\
 &=C_{E,\mathcal{K},\psi,n,p}\sum_{M=0}^\infty\lambda^{\frac{\mu-1}{n}}2^{-\frac{M(n-3)(1-\mu)}{n-2\mu}}.
\end{align*}
If $n>3$, the sum above converges. If $n=3$, we shall use \eqref{DLp>2} instead of the convex combination to give the $L^p$ estimate.
\begin{align*}
\norm{T_{j,k}^{l_1}}_{L^p}&\leq C_{E,\mathcal{K},\psi,n,p}\min\left\{\lambda^{\left(\frac{2\mu}{3}-1\right)\cdot \frac{1}{p}}\left(2^{-l_1}\right)^{-\frac{1}{p}}\left(2^{(\mu-1)l_1}\right)^{1-\frac{2}{p}},\quad 2^{(\mu-1)l_1}\right\}\\
&=C_{E,\mathcal{K},\psi,n,p}\min\left\{\lambda^{\left(\frac{2\mu}{3}-1\right)\cdot \frac{1}{p}}\left(2^{l_1}\right)^{\frac{1}{p}+(\mu-1)(1-\frac{2}{p})},\quad 2^{(\mu-1)l_1}\right\}.
\end{align*}
Hence
\begin{align*}
 \norm{\sum_{l_1\gg k}T_{j,k}^{l_1}}_{L^p}&\leq C_{E,\mathcal{K},\psi,n,p}\sum_{l_1=0}^\infty \min\left\{\lambda^{\left(\frac{2\mu}{3}-1\right)\cdot \frac{1}{p}}\left(2^{l_1}\right)^{\frac{1}{p}+(\mu-1)(1-\frac{2}{p})},\quad 2^{(\mu-1)l_1}\right\}\\
 &\leq C_{E,\mathcal{K},\psi,n,p}\lambda^{\frac{\mu-1}{3}}.
\end{align*}
For $p<2$, from \eqref{DLp<2}, we know that
\begin{align*}
\norm{T_{j,k}^{l_1}}_{L^p}&\leq C_{E,\mathcal{K},\psi,n,p}\lambda^{\left(\frac{2\mu}{n}-1\right)\cdot \frac{1}{p'}}\left(2^{-k(n-2)-M}\right)^{-\frac{1}{p'}}\left(2^{(\mu-1)(k+M)}\right)^{\frac{2}{p}-1},\\
\norm{T_{j,k}^{l_1}}_{L^p}&\leq C_{E,\mathcal{K},\psi,n,p}2^{(\mu-1)(k+M)}.
\end{align*}
Similarly, by convex combination, we know that for $0\leq \theta\leq 1$ it follows
\begin{align*}
\norm{T_{j,k}^{l_1}}_{L^p}\leq C_{E,\mathcal{K},\psi,n,p}\left[\lambda^{\left(\frac{2\mu}{n}-1\right)\cdot \frac{1}{p'}}\left(2^{-k(n-2)-M}\right)^{-\frac{1}{p'}}\left(2^{(\mu-1)(k+M)}\right)^{\frac{2}{p}-1}\right]^\theta\cdot\left[2^{(\mu-1)(k+M)}\right]^{1-\theta}.
\end{align*}
Again, choose suitable $\theta$ to eliminate $k$, this requires
\begin{align*}
\theta\left[\frac{n-2}{p'}+(\mu-1)\left(\frac{2}{p}-1\right)\right]+(\mu-1)(1-\theta)=0.
\end{align*}
It equals
\begin{equation*}
\theta=\frac{p'(1-\mu)}{n-2\mu}.
\end{equation*}
Given that $0\leq\theta\leq 1$, we can obtain $p'\leq \frac{n-2\mu}{1-\mu}$, i.e., $p\geq \frac{n-2\mu}{n-1-\mu}$. Therefore
\begin{align*}
 \norm{\sum_{l_1\gg k}T_{j,k}^{l_1}}_{L^p}&\leq C_{E,\mathcal{K},\psi,n,p}\sum_{M=0}^\infty\lambda^{\frac{2\mu-n}{np'}\cdot\theta}2^{M\left[\frac{1}{p'}+(\mu-1)\left(\frac{2}{p}-1\right)\right]\cdot\theta}2^{(\mu-1)(1-\theta)M}\\
 &=C_{E,\mathcal{K},\psi,n,p}\sum_{M=0}^\infty\lambda^{\frac{\mu-1}{n}}2^{-\frac{M(n-3)(1-\mu)}{n-2\mu}}.
\end{align*}
This sum converges if $n>3$, next we will treat the case $n=3$ as we have done previously.
\begin{align*}
\norm{T_{j,k}^{l_1}}_{L^p}&\leq C_{E,\mathcal{K},\psi,n,p}\min\left\{\lambda^{\left(\frac{2\mu}{n}-1\right)\cdot \frac{1}{p'}}\left(2^{-l_1}\right)^{-\frac{1}{p'}}\left(2^{(\mu-1)l_1}\right)^{\frac{2}{p}-1},\quad 2^{(\mu-1)l_1}\right\}\\
&=C_{E,\mathcal{K},\psi,n,p}\min\left\{\lambda^{\left(\frac{2\mu}{3}-1\right)\cdot \frac{1}{p'}}\left(2^{l_1}\right)^{\frac{1}{p'}+(\mu-1)(\frac{2}{p}-1)},\quad 2^{(\mu-1)l_1}\right\}.
\end{align*}
Then we have
\begin{align*}
\norm{\sum_{l_1\gg k}T_{j,k}^{l_1}}_{L^p}&\leq C_{E,\mathcal{K},\psi,n,p}\sum_{l_1=0}^\infty \min\left\{\lambda^{\left(\frac{2\mu}{3}-1\right)\cdot \frac{1}{p'}}\left(2^{l_1}\right)^{\frac{1}{p'}+(\mu-1)(\frac{2}{p}-1)},\quad 2^{(\mu-1)l_1}\right\}\\
&\leq C_{E,\mathcal{K},\psi,n,p}\lambda^{\frac{\mu-1}{3}}.
\end{align*}
Now we turn to deal with $\sum_{l_1\sim j\sim k}T_{j,k}^{l_1}$ in \eqref{T_Decom}. We decompose $T_{j,k}^{l_1}$ dyadically according to the second variety $y-\alpha_2x=0$, specifically we write
\begin{align*}
T_{j,k}^{l_1}f(x)&=\sum_{l_2}\int_{\R} e^{i\lambda S(x,y)}K(x,y)\left[1-\phi\left((x-y)\lambda^{\frac{1}{n}}\right)\right]\phi_{l_1}(y-x)\cdot\\
&\quad\quad\quad\quad\phi_{l_2}(y-\alpha_2x)\phi_k(x)\phi_k(y)\psi(x,y)f(y)dy\\
&:=\sum_{l_2}T_{j,k}^{l_1,l_2}f(x).
\end{align*}
Similarly, we shall also consider the relation between $l_2$ and $k$ and divide this sum into two parts
\begin{equation*}
T_{j,k}^{l_1}f(x)=\sum_{l_2\sim k}T_{j,k}^{l_1,l_2}f(x)+\sum_{l_2\gg k}T_{j,k}^{l_1,l_2}f(x).
\end{equation*}
For each operator in the latter sum, we can see that in the support of $T_{j,k}^{l_1,l_2}$, if we set $l_2=k+M$, we have
\begin{align*}
\notag&\nu=C_{\mathcal{K}}2^{-k(n-2)}2^{-M}, &\delta_1\approx2^{-k-M};\\
&C_K=C_{E,\mathcal{K},\psi}2^{\mu k},&\delta_2\approx2^{-k-M}.
\end{align*}
Repeating the above process, we can also conclude the $L^p$ estimate for $\sum_{l_2\gg k}T_{j,k}^{l_1,l_2}$. We omit the details. Thus we are left with $\sum T_{j,k}^{l_1,l_2}$ where $l_2\sim l_1\sim j \sim k$. Continue to decompose this operator and repeat the above process for the case $l_m\gg k(m\geq 3)$ until we are left with the last sum
\begin{equation*}
\sum_{l_s\sim\cdots\sim l_1\sim j\sim k}T_{j,k}^{l_1,\cdots,l_s}.
\end{equation*}
Notice that there are only finite operators among this sum, so we can consider only one such operator. In fact, in the support of $T_{j,k}^{l_1,\cdots,l_s}$, we have
\begin{align*}
\notag&\nu=C_{\mathcal{K}}2^{-k(n-2)}, &\delta_1\approx2^{-k};\\
&C_K= C_{E,\mathcal{K},\psi}2^{\mu k},&\delta_2\approx2^{-k}.
\end{align*}
By convex combination between the oscillatory estimate and the size estimate, we can get the desired result \eqref{T_DelEnd}.\\
\emph{Case II}: There is no $l$ such that $\alpha_l=1$.\\
Similar to the previous argument, we begin with giving a refined decomposition for the first variety, then keep the orthogonal part
\begin{equation*}
\sum_{l_1\sim j\sim k}T_{j,k}^{l_1}f(x),
\end{equation*}
while the other part
\begin{equation*}
\sum_{l_1\gg j\sim k}T_{j,k}^{l_1}f(x).
\end{equation*}
can be dealt with by summing all local $L^p$ estimates.  Regarding the orthogonal part as the initial operator and dyadically decompose it according to the second variety. Continue the process until the $s$-th step, we now come to the operator
\begin{equation*}
\sum_{l_s\sim\cdots\sim l_1\sim j\sim k}T_{j,k}^{l_1,\cdots,l_s}.
\end{equation*}
In the support of this operator, we can see that
\begin{equation*}
\abs{y-\alpha_1x}\approx C_{\mathcal{K}}2^{-k},\cdots, \abs{y-\alpha_sx}\approx C_{\mathcal{K}}2^{-k},
\end{equation*}
however, the singular kernel $K(x,y)$ is possible to vanish. So we shall provide another dyadic decomposition for this operator as follows
\begin{equation*}
T_{j,k}^{l_1,\cdots,l_s;t}=\sum_{t\sim l_s\sim\cdots\sim l_1\sim j\sim k}T_{j,k}^{l_1,\cdots,l_s;t}+\sum_{t\gg l_s\sim\cdots\sim l_1\sim j\sim k}T_{j,k}^{l_1,\cdots,l_s;t}.
\end{equation*}
In the support of each operator among the latter sum, we have
\begin{align*}
\notag&\nu=C_{\mathcal{K}}2^{-k(n-2)}, &\delta_1\approx 2^{-t};\\
&C_K=C_{E,\mathcal{K},\psi}2^{\mu t},&\delta_2\approx 2^{-t}.
\end{align*}
Repeating the previous process, we obtain
\begin{equation*}
\norm{\sum_{t\gg l_s\sim\cdots\sim l_1\sim j\sim k}T_{j,k}^{l_1,\cdots,l_s;t}}_{L^p}\leq C_{E,\mathcal{K},\psi,n,p}\lambda^{\frac{\mu-1}{n}}.
\end{equation*}
Now, we are left with the finite sum $\sum_{t\sim l_s\sim\cdots\sim l_1\sim j\sim k}T_{j,k}^{l_1,\cdots,l_s;t}$. By almost-orthogonality, we focus only on one such operator. In fact, in the support of one such operator, we have
\begin{align*}
\notag&\nu=C_{\mathcal{K}}2^{-k(n-2)}, &\delta_1\approx 2^{-k};\\
&C_K=C_{E,\mathcal{K},\psi}2^{\mu k},&\delta_2\approx 2^{-k}.
\end{align*}
This is same with the last part of \emph{Case I}, thus we are done with the proof.

\section{Proof of Theorem \ref{Main-Result}}

Now, we are ready to give the proof of Theorem \ref{Main-Result} since all preparation works have been finished. Firstly, we use Stein's interpolation between \eqref{dampL2-2} and \eqref{EndPo-2} and consequently have
\begin{equation}
\|T_Yf\|_{L^{\frac{n-2\mu}{n-1-\mu}}}\leq C_{E,\mathcal{K},\psi,n,p}\lambda^{-\frac{1-\mu}{n}}\|f\|_{L^{\frac{n-2\mu}{n-1-\mu}}}.\label{T_YEnd}
\end{equation}
Secondly, we apply Lemma \ref{InterpolationLem}, in which setting $p_0=2$, $a=n-1-\mu$, to \eqref{EndPo-1} and \eqref{dampL2-2} and consequently have
\begin{equation}
\|T_Xf\|_{L^{\frac{n-2\mu}{n-1-\mu}}}\leq C_{E,\mathcal{K},\psi,n,p}\lambda^{-\frac{1-\mu}{n}}\|f\|_{L^{\frac{n-2\mu}{n-1-\mu}}}.\label{T_XEnd}
\end{equation}
Finally, \eqref{T_YImed}, \eqref{T_XImed}, \eqref{T_DelEnd}, \eqref{T_YEnd}, \eqref{T_XEnd} together imply \eqref{Main-Inq} for $p<2$. The routine adjoint argument shall give the corresponding result for $p>2$.

\section{Further argument}
The previous argument is based on our additional condition \eqref{AddCondi}. However, in \cite{liu1999model} there is no such a condition. With the removal of \eqref{AddCondi}, we establish a result, similar to Theorem \ref{Main-Result}, as follows.
\begin{theorem}\label{Im_T}
We assume the same assumptions with Theorem \ref{Main-Result} except \eqref{AddCondi}, then
\begin{equation}\label{Main-inq-2}
\|Tf\|_{L^p}\leq C_{E,S,\psi,\mu,n,p}\lambda^{-\frac{1-\mu}{n}}\|f\|_{L^p}
\end{equation}
holds for $\frac{n-2\mu}{n-1-\mu}< p< \frac{n-2\mu}{1-\mu}$. Furthermore, for the endpoint $p=\frac{n-2\mu}{n-1-\mu}$, we also have the nearly sharp decay estimate
\begin{equation}\label{EndPo}
\|Tf\|_{L^{\frac{n-2\mu}{n-1-\mu}}}\leq C_{E,S,\psi,\mu,n,p}\lambda^{-\frac{1-\mu}{n}}\abs{\log_2(\lambda)}^{\frac{2-2\mu}{n-2\mu}}\|f\|_{L^{\frac{n-2\mu}{n-1-\mu}}}.
\end{equation}
\end{theorem}
Recall the strategy to prove Theorem \ref{Main-Result}: we first split the operator $T$ into three parts $T_X,T_\Delta,T_Y$, the next step is to use damped estimates to deal with $T_X$ and $T_Y$, finally we use local interpolation to treat $T_\Delta$. In the second step, the success of establishing the sharp $L^2$ estimates \eqref{dampL2-1}, \eqref{dampL2-2} relies on Lemma \ref{Orth_2}
and Remark \ref{DualLemma}. Here, the removal of \eqref{AddCondi} results in the failure of Remark \ref{DualLemma}. Instead, we make use of Remark \ref{OpVanDerCorput} to give the nearly sharp $L^2$ estimates as follows.
\begin{theorem}\label{Log-dampL2}
If the Hessian of the phase function is of the form \eqref{Hessian}, we have
\begin{align}
&\norm{T_Y^zf}_{L^2}\leq C_{E,\mathcal{K},\psi,n,z}\lambda^{\frac{\mu}{n}-\frac{1}{2}}\log(\lambda)\norm{f}_{L^2},\label{Log-DampL2-1}\\
&\norm{T_X^zf}_{L^2}\leq C_{E,\mathcal{K},\psi,n,z}\lambda^{\frac{\mu}{n}-\frac{1}{2}}\log(\lambda)\norm{f}_{L^2}.\label{Log-DampL2-2}
\end{align}
\end{theorem}
\begin{proof}
We still use the notations of \eqref{DX} and \eqref{DY}. If $\Re(z)=\frac{1}{2}$, on one hand, from Remark \ref{OpVanDerCorput} we know that
\begin{align}
\norm{D^X_{j}}_{L^2}&\leq C_{E,\mathcal{K},\psi,n,z}\lambda^{\frac{\mu}{n}-\frac{1}{2}},\label{DX-OscEst}\\
\norm{D^Y_{k}}_{L^2}&\leq C_{E,\mathcal{K},\psi,n,z}\lambda^{\frac{\mu}{n}-\frac{1}{2}},\label{DY-OscEst}
\end{align}
where $C_{E,A,B,\mu,n,z}$ is a constant with at most polynomial growth in $z$. On the other hand, we have the trivial size estimate
\begin{align}
\norm{D^X_{j}}_{L^2}&\leq C_{E,\mathcal{K},\psi}\left[2^{-j(n-2)}\right]^{\frac{1}{2}}2^{j\mu}2^{-j}=2^{-j\left(\frac{n}{2}-\mu\right)},\label{DX-SizEst}\\
\norm{D^Y_{k}}_{L^2}&\leq C_{E,\mathcal{K},\psi}\left[2^{-k(n-2)}\right]^{\frac{1}{2}}2^{k\mu}2^{-k}=2^{-k\left(\frac{n}{2}-\mu\right)}.\label{DY-SizEst}
\end{align}
Hence \eqref{DX}, \eqref{DX-OscEst} as well as \eqref{DX-SizEst} give that
\begin{align*}
\norm{T_X^z}_{L^2}=\norm{\sum_{j}D_j^X}_{L^2}&\leq \sum_{j}\norm{D_j^X}_{L^2}\\
&\leq C_{E,\mathcal{K},\psi,n,z}\sum_{j}\min\left\{\lambda^{\frac{\mu}{n}-\frac{1}{2}}, 2^{-j\left(\frac{n}{2}-\mu\right)}\right\}\\
&=C_{E,\mathcal{K},\psi,n,z}\lambda^{\frac{\mu}{n}-\frac{1}{2}}\log(\lambda).
\end{align*}
Similarly, for $T_Y^z$, there is
\begin{align*}
\norm{T_Y^z}_{L^2}=\norm{\sum_{k}D_k^Y}_{L^2}&\leq \sum_{k}\norm{D_k^Y}_{L^2}\\
&\leq C_{E,\mathcal{K},\psi,n,z}\sum_{k}\min\left\{\lambda^{\frac{\mu}{n}-\frac{1}{2}}, 2^{-k\left(\frac{n}{2}-\mu\right)}\right\}\\
&= C_{E,\mathcal{K},\psi,n,z}\lambda^{\frac{\mu}{n}-\frac{1}{2}}\log(\lambda).
\end{align*}
Therefore, we finish the proof of Theorem \ref{Log-dampL2}.
\end{proof}
By virtue of Theorem \ref{dampL2-1}, the proof of Theorem \ref{Im_T} is same with that of Theorem \ref{Main-Result} which we have shown in Section 7. Here, we omit the details for simplicity.

\section{A necessary condition}
In this part, we will give an example to demonstrate a necessary condition for $p$ to preserve the sharp decay. \\
If we assume that
\begin{align*}
&S(x,y)=x^{n-1}y+xy^{n-1}, & K(x,y)=\abs{x-y}^{-\mu};\\
&\lambda\geq 10, & \psi (x,y)\equiv 1 \text{ on } \abs{(x,y)}\leq \frac{1}{2}.
\end{align*}
Moreover, we assume that $f(y)=\chi_{[\frac{1}{4}, \frac{1}{2}]}(y)$, if $0\leq x\leq\frac{1}{100\lambda}$, we know that
\begin{align*}
\abs{Tf(x)}&\geq \abs{\int_{\frac{1}{4}}^{\frac{1}{2}}\cos \left(\frac{y }{1000^{n-1}}+\frac{y^{n-1}}{1000}\right)\abs{y-\frac{1}{1000}}^{-\mu}dy}\\
&\geq {\int_{\frac{1}{4}}^{\frac{1}{2}}\cos \left(\frac{1}{1000^{n-1}\cdot 2}+\frac{1}{1000\cdot 2^{n-1}}\right)\abs{\frac{1}{2}-\frac{1}{1000}}^{-\mu}dy}\\
&\geq \frac{1}{10}.
\end{align*}
If the apriori estimate $\|Tf\|_{L^p}\lesssim \lambda^{-\frac{1-\mu}{n}}\|f\|_{L^p}$ holds, we can see that
\begin{align*}
\lambda^{-\frac{1-\mu}{n}}\cdot \frac{1}{4^{\frac{1}{p}}}&\gtrsim \left[\int_{\R}\abs{\int_{\R} e^{i\lambda(x^{n-1}y+xy^{n-1})}\abs{x-y}^{-\mu}\psi(x,y)\chi_{[\frac{1}{4}, \frac{1}{2}]}(y)dy}^pdx\right]^{\frac{1}{p}}\\
&\geq \left[\int_{0}^{\frac{1}{100\lambda}}\abs{\int_{\R} e^{i\lambda(x^{n-1}y+xy^{n-1})}\abs{x-y}^{-\mu}\psi(x,y)\chi_{[\frac{1}{4}, \frac{1}{2}]}(y)dy}^pdx\right]^{\frac{1}{p}}\\
&\geq \left(\int_{0}^{\frac{1}{100\lambda}}\frac{1}{10^p}dx\right)^{\frac{1}{p}}\\
&=\lambda^{-\frac{1}{p}}\frac{1}{10^{10p+1}}.
\end{align*}
This implies
\begin{equation*}
\lambda^{-\frac{1-\mu}{n}}\gtrsim \lambda^{-\frac{1}{p}}.
\end{equation*}
Since $\lambda$ can be arbitrarily large, this requires
\begin{equation*}
p\leq \frac{n}{1-\mu}.
\end{equation*}
Interchanging the roles of $x$ and $y$, we can get the other necessary condition
\begin{equation*}
p\geq \frac{n}{n-1+\mu}.
\end{equation*}
\begin{remark}
Observe the necessary condition, there is a gap between it with our result.
\end{remark}

{\bf Acknowledgement:} The author would like to acknowledge financial support from Jiangsu Natural Science Foundation, Grant No. BK20200308. The author thanks Zuoshunhua Shi for many helpful suggestions.


\end{document}